\documentclass[12pt,a4paper]{article}
\usepackage{latexsym,amssymb,amsfonts,amsmath,amsthm,nccmath,float,enumitem,setspace,bm,authblk}
\usepackage[usenames,dvipsnames]{xcolor}
\usepackage[hidelinks]{hyperref}
\usepackage[margin=2cm]{geometry}

\setlist{topsep=3pt,partopsep=0pt,itemsep=1pt,parsep=0pt}

\hypersetup{
    colorlinks,
    linkcolor={red!80!black},
    citecolor={blue!80!black},
    urlcolor={blue!80!black}
}

\newtheorem{Theorem}{Theorem}

\newtheorem{Corollary}{Corollary}
\newtheorem{Conjecture}{Conjecture}

\newtheorem{Lemma}{Lemma}

\newcommand{\Z}{\mathbb{Z}}

\newcommand{\B}{\mathcal{B}}
\newcommand{\F}{\mathcal{F}}
\newcommand{\R}{\mathbb{R}}

\renewcommand{\P}{\mathcal{P}}
\newcommand{\Q}{\mathcal{Q}}
\renewcommand{\R}{\mathcal{R}}

\newcommand{\e}{\epsilon}

\renewcommand{\t}{\tau}
\renewcommand{\l}{\lambda}

\newcommand{\codeg}{\mathop{\mathrm{codeg}}\nolimits}

\def \dev {{\rm dev}}

\def \leq {\leqslant}
\def \geq {\geqslant}

\let\oldproofname=\proofname
\renewcommand{\proofname}{\rm\bf{\oldproofname}}

\begin{document}

\title{Nov\'{a}k's conjecture on cyclic Steiner triple systems and its generalization}

\author[a]{Tao Feng}
\author[b]{Daniel Horsley}
\author[c]{Xiaomiao Wang}
\affil[a]{Department of Mathematics, Beijing Jiaotong University, Beijing 100044, P.R. China}
\affil[b]{School of Mathematics, Monash University, VIC 3800, Australia}
\affil[c]{School of Mathematics and Statistics, Ningbo University, Ningbo 315211, P.R. China}

\renewcommand*{\Affilfont}{\small\it}
\renewcommand\Authands{ and }
\date{}

\maketitle
\footnotetext[1]{E-mail address: tfeng@bjtu.edu.cn; Supported by NSFC under Grant 11871095}
\footnotetext[2]{E-mail address: danhorsley@gmail.com; Supported by ARC grants DP150100506 and FT160100048}
\footnotetext[3]{E-mail address: wangxiaomiao@nbu.edu.cn; Supported by NSFC under Grant 11771227}
\begin{abstract}
Nov\'{a}k conjectured in 1974 that for any cyclic Steiner triple systems of order $v$ with $v\equiv 1\pmod{6}$, it is always possible to choose one block from each block orbit so that the chosen blocks are pairwise disjoint. We consider the generalization of this conjecture to cyclic $(v,k,\l)$-designs with $1 \leq \l \leq k-1$. Superimposing multiple copies of a cyclic symmetric design shows that the generalization cannot hold for all $v$, but we conjecture that it holds whenever $v$ is sufficiently large compared to $k$. We confirm that the generalization of the conjecture holds when $v$ is prime and $\l=1$ and also when $\l \leq (k-1)/2$ and $v$ is sufficiently large compared to $k$.
As a corollary, we show that for any $k \geq 3$, with the possible exception of finitely many composite orders $v$, every cyclic $(v,k,1)$-design without short orbits is generated by a $(v,k,1)$-disjoint difference family.
\end{abstract}


\section{Introduction}

Let $V$ be a set of $v$ {\em points}, and $\B$ be a collection of $k$-subsets
of $V$ called {\em blocks}. A pair $(V, \B)$ is called a {\em $(v,k,\lambda)$-design} if every pair of distinct elements of $V$ is contained in precisely $\lambda$ blocks of $\B$. A $(v,3,1)$-design is called a {\em Steiner triple system} of order $v$ and is written as an STS$(v)$.

An {\em automorphism} of a $(v,k,\lambda)$-design $(V, \B)$ is a permutation on $V$ leaving $\B$ invariant. A $(v,k,\lambda)$-design is said to be {\em cyclic} if it admits an automorphism
consisting of a cycle of length $v$. Without loss of generality we identify $V$ with ${\mathbb Z}_v$, the additive group of integers modulo $v$. The blocks of a cyclic $(v,k,\lambda)$-design can be partitioned into orbits under ${\mathbb Z}_{v}$. We can choose any fixed block from each orbit and then call these {\em base blocks}. If the cardinality of an orbit is equal to $v$, the orbit is {\em full}. Otherwise, it is {\em short}. It follows from the orbit-stabilizer theorem that the cardinality of any orbit is a divisor of $v$ and is at least $v/k$. If $\gcd(v,k)=1$, then all orbits of a cyclic $(v,k,\lambda)$-design are full (see \cite[Lemma 1]{k}). It is known that a cyclic STS$(v)$ exists if and only if $v\equiv 1,3\pmod{6}$ and $v\neq 9$ (see \cite[Theorem 7.3]{triplesystem}).

A useful tool for generating cyclic designs is the concept of a difference family. A {\em $(v,k,\l)$-cyclic difference family} is a family $\F$ of $k$-subsets (called {\em base blocks}) of ${\mathbb Z}_{v}$ such that the multiset $\Delta \F:=\{x-y:x,y\in F, x\not=y,F\in \F\}$ contains every element of ${\mathbb Z}_{v}\setminus \{0\}$ exactly $\lambda$ times. Such a family is denoted as a $(v,k,\lambda)$-CDF. It consists of $\lambda(v-1)/(k(k-1))$ base blocks. A $(v,k,\lambda)$-CDF $\F$ can generate a cyclic $(v,k,\lambda)$-design with block-multiset $\dev \F:=\{F+t:F\in \F,t\in \Z_v\}$ (see \cite[Theorem 3.46]{s}). Furthermore, when $\gcd(v,k)=1$, $\F$ is a $(v,k,\lambda)$-CDF if and only if $\dev \F$ is a cyclic $(v,k,\lambda)$-design (see \cite[Proposition VII.1.5]{bjl}).

A $(v,k,\lambda)$-CDF is said to be \emph{disjoint} and written as a $(v,k,\lambda)$-DDF when its base blocks are mutually disjoint. Nov\'{a}k \cite{novak} conjectured in 1974 that for any cyclic STS$(v)$ with $v\equiv 1\pmod{6}$, it is always possible to find a set of $(v-1)/6$ disjoint base blocks which come from different block orbits to form a $(v,3,1)$-DDF (see also \cite[Remark 16.22]{abel} or \cite[Work point 22.5.2]{triplesystem}).

\begin{Conjecture}\textup{(Nov\'{a}k, $1974$) \cite{novak}}\label{weak conjecture}
Every cyclic STS$(v)$ with $v\equiv 1\pmod{6}$ is generated by a $(v,3,1)$-DDF.
\end{Conjecture}

Conjecture \ref{weak conjecture} is widely believed to be true but not much progress has been made on it. So far it is only known that Conjecture \ref{weak conjecture} holds for all $v\equiv 1\pmod{6}$ and $v\leq 61$ (see \cite[Theorem 22.3]{triplesystem}). On the other hand, Dinitz and Rodney \cite{dinitzUtil} proved that a $(v,3,1)$-DDF exists for any $v\equiv 1 \pmod{6}$ by taking a suitable $(v,3,1)$-CDF and then replacing each of its base blocks $B_i$ by a suitable translate $B_i+t_i$. For more information on $(v,3,1)$-DDFs with $v\equiv 3 \pmod{6}$, interested readers are referred to \cite{buratti36,dinitz36}.

Recently, using the Combinatorial Nullstellensatz, Karasev and Petrov \cite{karasev} proved the following result.

\begin{Lemma}\textup{\cite[Theorem $2$]{karasev}}\label{pairwisedisjoint}
Let $\mathbb{F}$ be an arbitrary field, and let $m$ and $d$ be positive integers such that
$(md)!/(d!)^m\neq0$ in $\mathbb{F}$. Let $X_1,\ldots, X_m$ and $T_1,\ldots, T_m$ be subsets of $\mathbb{F}$ such
that
$$\forall i < j\ \  |X_i-X_j|\leq2d,\ \ \  \forall i\ \ |T_i|\geq (m-1)d+1,$$
where $X_i-X_j:=\{x-y:x\in X_i,y\in X_j\}$.
Then there exists a system of representatives $t_i\in T_i$ such that the sets
$X_1+t_1,\ldots, X_m+t_m$ are pairwise disjoint.
\end{Lemma}

We now apply Lemma \ref{pairwisedisjoint} to show that Conjecture \ref{weak conjecture} is true whenever $v$ is a prime.

\begin{Theorem}\label{resultprime}
Let $k\geq 2$ and let $p$ be a prime. Every cyclic $(p,k,1)$-design is generated by a $(p,k,1)$-DDF.
\end{Theorem}

\proof We may assume $p>k$ because otherwise the result is trivial. Since $\gcd(p,k)=1$, a cyclic $(p,k,1)$-design has $m=(p-1)/(k(k-1))$ full orbits and no short orbits. Let $B_1,\ldots, B_m$ be base blocks of a cyclic $(p,k,1)$-design and let $d=\lceil k^2/2\rceil$. Then $|B_i-B_j|\leq2d$ for any $1\leq i<j\leq m$. Let $T_1=\cdots=T_m=\mathbb Z_p$. Then $|T_i|=p\geq(m-1)d+1$ for $k\geq2$. Since $md<p$ when $k\geq2$, $(md)!/(d!)^m\not\equiv 0 \pmod{p}$. Therefore, by Lemma \ref{pairwisedisjoint}, there exists a system of representatives $t_i\in T_i$ such that $B_1+t_1,\ldots, B_m+t_m$ are pairwise disjoint. So $B_1+t_1,\ldots, B_m+t_m$ form a $(p,k,1)$-DDF. \qed

Theorem \ref{resultprime} motivates us to present the following conjecture on cyclic $(v,k,1)$-designs, which also allows for designs with short orbits.

\begin{Conjecture}\label{conj:any k}
For any cyclic $(v,k,1)$-design, it is always possible to choose one block from each block orbit so that the chosen blocks are pairwise disjoint.
\end{Conjecture}

The existence of $(v,k,1)$-DDFs is in general quite a hard problem. Conjecture \ref{conj:any k}, if true, would reduce the existence of $(v,k,1)$-DDFs to the existence of $(v,k,1)$-CDFs. The following results on CDFs are known in the literature.

\begin{Lemma}\label{lem:456-wilson}\hspace{1mm}
\begin{itemize}
\item[$(1)$]\textup{\cite{chen45}} For any prime $p\equiv 1\pmod{12}$, there exists a $(p, 4, 1)$-CDF.
\item[$(2)$]\textup{\cite{chen45}} For any prime $p\equiv 1\pmod{20}$, there exists a $(p, 5, 1)$-CDF.
\item[$(3)$]\textup{\cite{chen6}} For any prime $p\equiv 1\pmod{30}$ and $p\neq 61$, there exists a $(p, 6, 1)$-CDF.
\item[$(4)$]\textup{\cite{bp}} Let $p\equiv 1\pmod{k(k-1)}$ be a prime. Then a $(p,k,1)$-CDF exists if $p>\binom{k}{2}^{2k}$.
\end{itemize}
\end{Lemma}

As a corollary of Theorem \ref{resultprime} together with Lemma \ref{lem:456-wilson}, we obtain the following existence results on DDFs.

\begin{Theorem}
Let $p\equiv 1\pmod{k(k-1)}$ be a prime.
\begin{itemize}
\item[$(1)$] There exists a $(p,k,1)$-DDF for each $k\in\{4,5,6\}$ and $(k,p)\neq (6,61)$.
\item[$(2)$] There exists a $(p,k,1)$-DDF whenever $p>\binom{k}{2}^{2k}$.
\end{itemize}
\end{Theorem}

We remark that by using Weil's theorem on estimates of character sums, Wu, Yang and Huang \cite{wuActa} also established the existence of a $(p,4,1)$-DDF for any prime $p\equiv 1\pmod{12}$. We also observe that the main result of \cite{EhaGloJoo} shows that, for fixed $k$ and large $v$, one can find a family $\mathcal{F}$ of $(1-o(1))\frac{v-1}{k(k-1)}$ pairwise disjoint base blocks of size $k$ such that $\Delta \mathcal{F}$ contains each difference at most once. This is accomplished by letting $H$ be the disjoint union of $(1-o(1))\frac{v-1}{k(k-1)}$ copies of $K_k$ and applying \cite[Theorem 1.2]{EhaGloJoo} to find a rainbow copy of $H$ in the complete graph on $\Z_v$ with edges coloured according to their differences.

In this paper, we shall provide a proof of Conjecture \ref{conj:any k} when $v$ is sufficiently large compared to $k$. In fact, we consider a more general setting. We shall examine cyclic $(v,k,\lambda)$-designs with $k \geq 2\lambda+1$. As the main result of this paper, we prove Theorem~\ref{thm:main theorem} below. In fact we prove a stronger statement which sometimes guarantees the existence of a family of mutually disjoint blocks containing many blocks from each orbit (see Theorem~\ref{thm:asympt}).

\begin{Theorem}\label{thm:main theorem}
Let $k$ and $\lambda$ be fixed positive integers such that $k\geq 2\l+1$. There exists an integer $v_0$ such that, for any cyclic $(v,k,\lambda)$-design with $v\geq v_0$, it is always possible to choose one block from each block orbit so that the chosen blocks are pairwise disjoint.
\end{Theorem}

Combining Theorems~\ref{resultprime} and \ref{thm:main theorem} yields the following corollary.

\begin{Corollary}
Let $k \geq 3$ be a fixed integer. With the possible exception of finitely many composite orders $v$, every cyclic $(v,k,1)$-design without short orbits is generated by a $(v,k,1)$-DDF.
\end{Corollary}

\section{Preliminaries}

For any positive integer $c$, let $[c]:=\{1,\ldots,c\}$. We will make use of the following simple lemma which shows that, for large $v$ and fixed $k$ and $\l$, a cyclic $(v,k,\l)$-design has few short orbits.

\begin{Lemma}\label{L:fewShort}
Let $k \geq 2$ and $\lambda \geq 1$ be fixed integers. If $(V,\B)$ is a cyclic $(v,k,\l)$-design with $h$ short orbits and $m$ full orbits, then
\begin{itemize}[itemsep=1mm]
    \item[\textup{(i)}]
$h \leq 2\l\sqrt{k}$; and
    \item[\textup{(ii)}]
$\frac{\l(v-1)}{k(k-1)}-2\l\sqrt{k} \leq m \leq \frac{\l(v-1)}{k(k-1)} \leq m+h \leq \frac{\l(v-1)}{k(k-1)}+2\l\sqrt{k}$.
\end{itemize}
\end{Lemma}

\begin{proof}
Let the point set of $(V,\B)$ be $\Z_v$ and let $\B_1,\ldots,\B_{h}$ be the short orbits of $(V,\B)$. Let $i \in [h]$. Recall that by the orbit-stabilizer theorem we have $|\B_i|=\ell_i$ where $\frac{v}{k} \leq \ell_i<v$ and $\ell_i \mid v$. Let $B_i$ be a base block from $\B_i$ such that $B_i$ contains the point $0$. Since $|\B_i|=\ell_i$, $B_i+\ell_i=B_i$. It follows that $B_i$ contains all multiples of $\ell_i$. Write $S_i:=\{0,\ell_i,2\ell_i,\ldots,(\frac{v}{\ell_i}-1)\ell_i\}$. Then $S_i\subseteq B_i$. Furthermore, for any $a\in B_i$, $a+S_i\subseteq B_i$, and so $B_i$ is a disjoint union of some cosets of $S_i$ in $\mathbb Z_v$, which implies $|S_i|\mid |B_i|$. That is, $\frac{v}{\ell_i}\mid k$. Also, because exactly $\l$ blocks in $\B$ contain the pair $\{0,\ell_i\}$, we have that at most $\lambda$ of the orbits $\B_1,\ldots,\B_h$ have cardinality $\ell_i$.

Thus, $h \leq \lambda \sigma_0(k)$ where $\sigma_0(k)$ denotes the number of divisors of $k$. We know that $\sigma_0(k) \leq 2\sqrt{k}$ for any positive integer $k$ by using the fact that $d\mid k$ if and only if $\frac{k}{d}\mid k$, and so (i) follows. Then (ii) follows from (i) by routine calculation after observing that $mv+\sum_{i=1}^h\ell_i=|\B|=\frac{\l v(v-1)}{k(k-1)}$.
\end{proof}

An {\em $r$-uniform hypergraph} $G$ is a pair $(V, E)$ where $V$ is a vertex set and $E$ is a set of $r$-subsets of $V$ called edges. The {\em degree} $\deg_G(x)$ of a vertex $x\in V$ is the number of edges of $G$ containing $x$. For distinct vertices $x$ and $y$ of $G$, the {\em codegree} $\codeg_G(x, y)$ is the number of edges of $G$ containing both $x$ and $y$. We write $\delta_G:=\min\limits_{x\in V} \deg_{G}(x)$, $\Delta_G:=\max\limits_{x\in V} \deg_{G}(x)$ and $\Delta^c_G:=\max\limits_{x,y\in V,x\neq y} \codeg_G(x,y)$.

A {\em proper edge-colouring} of a hypergraph $G=(V,E)$ with $c$ colours is a function $f: E\longrightarrow [c]$ such that no two edges that share a vertex get the same colour. The following powerful result of Pippenger and Spencer \cite{PipSpe} (based on the R\"{o}dl nibble) shows that every almost regular $r$-uniform hypergraph $G$ with small maximum codegree can be edge-coloured with close to $\Delta_G$ colours.

\begin{Lemma}\label{thm:pip_sp} \textup{\cite{PipSpe}}
Let $r\geq 2$ be an integer. For each real number $\eta>0$, there exists a real number $\eta^*>0$ and an integer $n_0$ such that if $G$ is an $r$-uniform hypergraph on $n\geq n_0$ vertices satisfying $\delta_G\geq (1-\eta^*)\Delta_G$ and $\Delta^c_G \leq \eta^* \Delta_G$, then $G$ has a proper edge-colouring with $(1+\eta)\Delta_G$ colours.
\end{Lemma}

\section{Proof of Theorem \ref{thm:main theorem}}

A {\em partial parallel class} of a $(v,k,\lambda)$-design is a set of pairwise disjoint blocks. Let $(V,\B)$ be a cyclic $(v,k,\lambda)$-design with orbits $\B_1,\ldots,\B_t$, let $\P$ be a partial parallel class of $(V,\B)$ and let $s=\lfloor\frac{k-1}{\lambda}\rfloor$. For any nonnegative integer $a$ we define $T_a(\P)=\{i \in [t]: |\P \cap \B_i|=a\}$ to be the set of indices of orbits of $(V,\B)$ that contain exactly $a$ blocks in $\P$, and we define $\t_a(\P)=|T_a(\P)|$. Also, we say that a block $B \in \B$ is \emph{$\P$-good} if, for each $i \in [t]$, $B$ intersects at most one block in $\P \cap \B_i$ and, for each $i \in T_0(\P) \cup \cdots \cup T_{s-1}(\P)$, $B$ intersects no block in $\P \cap \B_i$. Blocks in $\B$ that are not $\P$-good are \emph{$\P$-bad}. Intuitively, a $\P$-good block $B$ has the property that if we add $B$ to $\P$ and remove all blocks of $\P$ incident with $B$, then each orbit that intersected $\P$ in at least $s-1$ blocks still intersects the resulting partial parallel class in at least $s-1$ blocks. Finally we define, if $s \geq 2$,
\[d(\P)=\medop\sum_{a=0}^{s-2}(s-1-a)\t_a(\P).\]
One can think of $d(\P)$ as a measure of how far $\P$ is from intersecting each orbit in at least $s-1$ blocks. The definitions of $\P$-good and $d(\P)$ are implicitly dependent on the value of $s=\lfloor\frac{k-1}{\lambda}\rfloor$.

Our strategy is to first, in Lemma~\ref{L:applyPipSpe} below, apply Lemma~\ref{thm:pip_sp} to an auxiliary hypergraph in order to obtain a partial parallel class in the design that contains $s$ blocks from almost every orbit. For such a partial parallel class $\P$ we then, in Lemma~\ref{L:greedy}, prove that if each orbit that intersects $\P$ in fewer than $s-1$ blocks contains sufficiently many $\P$-good blocks, then $\P$ can be modified to produce a new class that contains $s$ blocks from almost every orbit and $s-1$ blocks from each remaining orbit. Finally, to prove Theorem~\ref{thm:asympt}, we show that Lemma~\ref{L:greedy} can successfully be applied to a partial parallel class obtained by making some modifications to a class given by Lemma~\ref{L:applyPipSpe}.

\begin{Lemma}\label{L:applyPipSpe}
Let $k$ and $\l$ be positive integers and let $s=\lfloor\frac{k-1}{\l}\rfloor$. For each real number $\e^*>0$, there exists an integer $v^*_0$ such that, for each integer $v \geq v^*_0$, any cyclic $(v,k,\l)$-design with $m$ full orbits has a partial parallel class $\P$ that contains $s$ blocks from each of $(1-\e^*)m$ full orbits and no blocks from any other orbit.
\end{Lemma}

\begin{proof}
Let $(V,\B)$ be a cyclic $(v,k,\l)$-design with full orbits $\B_1,\ldots,\B_m$. Observe that, by Lemma~\ref{L:fewShort}(ii), $\frac{\l(v-1)}{k(k-1)}-2\l\sqrt{k} \leq m \leq \frac{\l(v-1)}{k(k-1)}$. Hence, supposing $v$ is sufficiently large, we have $\frac{\l(v-1)}{k^2} < m \leq \frac{\l(v-1)}{k(k-1)}$. Let $w=\lfloor\frac{v-1}{k}\rfloor$ and choose integers $s_1,\ldots,s_m \in \{s,s+1\}$ such that $s_1+\cdots+s_m =w$. Such integers exist because $sm \leq \frac{v-1}{k}$ using $s \leq \frac{k-1}{\l}$ and $m \leq \frac{\l(v-1)}{k(k-1)}$, and because $(s+1)m > \frac{v-1}{k}$ using $s+1 \geq \frac{k}{\l}$ and $m > \frac{\l(v-1)}{k^2}$. Let $W=\{u_{i,j}: i \in [m],j \in [s_i]\}$ be a set of $w$ vertices disjoint from $V$. We form a $(k+1)$-uniform hypergraph $G$ with vertex set $V \cup W$ and edge set
$$\{B \cup \{u_{i,j}\}: B \in \B_i, i \in [m], j \in [s_i]\}.$$

Observe that, for each $x \in V$, we have $\deg_G(x)=ks_1+\cdots+ks_m=kw$ because $x$ is in $k$ blocks in each full orbit, and hence we have  $v-k \leq \deg_G(x) \leq v-1$. Also, $\deg_G(x)=v$ for each $x \in W$ because each full orbit contains $v$ blocks. Furthermore $\codeg_G(x,y) \leq \lambda(s+1) \leq k+\l-1$ for all distinct $x,y \in V$ because $(V,\B)$ is a design of index $\lambda$, $\codeg_G(x,y)=0$ for all distinct $x,y \in W$, and $\codeg_G(x,y) = k$ for all $x \in V$ and $y \in W$ because $k$ blocks from any full orbit contain a given vertex in $V$. So $G$ has $v+w$ vertices, $vw$ edges, $\delta_G \geq v-k$, $\Delta_G \leq v$, and $\Delta^c_G \leq k+\l-1$. Thus Lemma \ref{thm:pip_sp} implies that for any real number $\e^* > 0$, supposing $v$ is sufficiently large, $G$ has a proper edge-colouring with $(1+\frac{\e^*}{s+1})v$ colours.

Let $\mathcal{C}$ be a largest colour class of this colouring. Then $\mathcal{C}$ is a set of disjoint edges of $G$ and, because $G$ has $vw$ edges, $|\mathcal{C}| \geq \frac{(s+1)w}{s+1+\e^*}>w-\e^*\frac{w}{s+1}>w-\e^* m$ where the last inequality follows because $w < (s+1)m$. Let
\[M=\bigl\{i \in [m]: |\{j \in [s_i]: \text{$u_{i,j}$ is in an edge in $\mathcal{C}$}\}| \geq s\bigr\}.\]
Observe that $|M| > (1-\e^*)m$ because each edge of $G$ contains exactly one vertex in $W$ and hence there are less than $\e^* m$ vertices in $W$ that are not in an edge of $\mathcal{C}$. Let $\mathcal{C}'$ be the set of edges in $\mathcal{C}$ that contain a vertex in $\{u_{i,j}:i \in M, j \in [s_i]\}$ and let $\P=\{E \cap V: E \in \mathcal{C}'\}$. Then, by the definitions of $G$ and $\mathcal{C}'$, $\P$ is a partial parallel class in $(V,\B)$ that contains at least $s$ blocks from $\B_i$ for each $i \in M$ and no other blocks. So, by deleting some blocks from $\P$ if necessary, we can obtain a partial parallel class with the desired properties.
\end{proof}

\begin{Lemma}\label{L:greedy}
Let $k$ and $\l$ be positive integers such that $k \geq 2\l+1$ and let $s=\lfloor\frac{k-1}{\l}\rfloor$. Let $(V,\B)$ be a cyclic \hbox{$(v,k,\l)$}-design with orbits $\B_1,\ldots,\B_t$, and let $\P'$ be a partial parallel class that contains at most $s$ blocks from each orbit. If, for each $i \in T_0(\P') \cup \cdots \cup T_{s-2}(\P')$, $\B_i$ contains more than $k^2(ks-k+1)(d(\P')-1)$ $\P'$-good blocks, then there is a partial parallel class $\P''$ of $(V,\B)$ such that $\t_{s-1}(\P'') \leq (k+1)d(\P')+\t_{s-1}(\P')$ and $\t_s(\P'') = t-\t_{s-1}(\P'')$.
\end{Lemma}

\begin{proof}
Note that $s \geq 2$ by our hypotheses. We prove the result by induction on the quantity $d(\P')$.
If $d(\P')=0$ then $\t_0(\P')=\cdots=\t_{s-2}(\P')=0$ and we can take $\P''=\P'$ to complete the proof. So suppose that $d(\P')=\ell$ for some positive integer $\ell$ and that the result holds for $d(\P') < \ell$. Let $j \in T_0(\P') \cup \cdots \cup T_{s-2}(\P')$ and let $B_j$ be a $\P'$-good block in $\B_j$ (such a block exists by the hypotheses of the lemma because $d(\P') \geq 1$). Let $\Q$ be the set of blocks in $\P'$ which intersect $B_j$ and let $\P^*=(\P' \cup \{B_j\}) \setminus \Q$. Note that $\P^*$ is a partial parallel class of $(V,\B)$ and that $|\Q| \leq |B_j|=k$. Also, because $B_j$ was $\P'$-good, $\t_1(\Q)=|\Q|$ and $T_1(\Q) \subseteq T_s(\P')$.

Observe that $|\P^* \cap \B_j|=|\P' \cap \B_j|+1$, $|\P^* \cap \B_i|=s-1$ for each $i \in T_1(\Q)$, and $\P^* \cap \B_i=\P' \cap \B_i$ for all $i \in [t] \setminus (T_1(\Q) \cup \{j\})$.
Thus $d(\P^*)=d(\P')-1$ and $\t_{s-1}(\P^*) \leq \t_{s-1}(\P')+k+1$.
Any block in $\B$ that was $\P'$-good but is $\P^*$-bad must intersect one of the at most $ks-k+1$ blocks in  $\{B_j\} \cup \bigcup_{i\in T_1(\Q)}(\P^*\cap\B_i)$. For each $i \in T_0(\P^*) \cup \cdots \cup T_{s-2}(\P^*)$, at most $k^2$ blocks in $\B_i$ intersect each of these blocks and so, because more than $k^2(ks-k+1)(d(\P')-1)$ blocks in $\B_i$ were $\P'$-good, more than
$$k^2(ks-k+1)(d(\P')-1)-k^2(ks-k+1)=k^2(ks-k+1)(d(\P')-2)=k^2(ks-k+1)(d(\P^*)-1)$$
blocks in $\B_i$ are $\P^*$-good. Thus we can apply our inductive hypothesis to $\P^*$ to establish the existence of a partial parallel class $\P''$ of $(V,\B)$ such that
$\t_{s-1}(\P'') \leq (k+1)d(\P^*)+\t_{s-1}(\P^*)$
and $\t_s(\P'') =t-\t_{s-1}(\P'')$. The proof is now complete by observing that
$$\t_{s-1}(\P'') \leq (k+1)d(\P^*)+\t_{s-1}(\P^*) \leq (k+1)(d(\P')-1)+\t_{s-1}(\P') + k+1=(k+1)d(\P')+\t_{s-1}(\P').$$ \qedhere
\end{proof}

\begin{Theorem}\label{thm:asympt}
Let $k$ and $\l$ be fixed positive integers such that $k \geq 2\l+1$ and let $s=\lfloor\frac{k-1}{\l}\rfloor$. For each real number $\epsilon>0$, there is an integer $v_0$ such that, for each integer $v \geq v_0$, any cyclic $(v,k,\l)$-design with $t$ orbits has a partial parallel class that contains $s-1$ blocks from each of at most $\e t$ orbits and contains $s$ blocks from each other orbit.
\end{Theorem}

\begin{proof}
Note that $s \geq 2$ by our hypotheses. We may assume that $\e < \frac{1}{4k^2}$. Let $\e^*=\frac{\e}{2(k+1)s}$. Let $(V,\B)$ be a cyclic $(v,k,\l)$-design with orbits $\B_1, \ldots, \B_t$ and suppose that $m$ of these orbits are full. Throughout this proof, we will tacitly assume $v$ is sufficiently large whenever necessary and will use asymptotic notation with respect to this regime. Note that $t=\frac{\l(v-1)}{k(k-1)}+O(1)$ by Lemma~\ref{L:fewShort}(ii) and hence $t=\Theta(v)$. By Lemma~\ref{L:applyPipSpe} there is a partial parallel class $\P$ of $(V,\B)$ such that $T_0(\P)$ contains at most $\e^*m \leq \epsilon^* t$ indices of full orbits and every other index of a full orbit is in $T_s(\P)$. Let
\[R = \{i \in [t]: \text{$\B_i$ contains at least $\tfrac{1}{2}st$ $\P$-bad blocks}\}.\]

A block in $\B$ is $\P$-bad if and only if it intersects at least two blocks in $\P \cap \B_i$ for some $i \in T_s(\P)$. At most $k^2 \l\binom{s}{2}$ blocks of $\B$ intersect at least two blocks in $\P \cap \B_i$ for each $i \in T_s(\P)$, and so it follows that at most $k^2 \l\binom{s}{2}\t_s(\P)\leq k^2 \l\binom{s}{2}t$ blocks in $\B$ are $\P$-bad. Thus, by the definition of $R$, we have $|R| \leq k^2s\l$.

We can greedily choose a partial parallel class $\R$ in $(V,\B)$ such that $|\R \cap \B_i|=s$ for each $i \in R$ and $\R \cap \B_i=\emptyset$ for each $i \in [t] \setminus R$. To see this, suppose that $x<s|R|\leq k^2s^2\l$ blocks of the class have already been chosen and note that, for each $i \in R$, at most $k^2$ of blocks in $\B_i$ intersect each already chosen block and \[|\B_i|  \geq \tfrac{v}{k} \gg k^4s^2\l >k^2x.\]
Thus we can indeed choose a suitable $\R$ greedily.

Now let
\[Q = \{i \in [t] \setminus R: \text{some block in $\P \cap \B_i$ intersects some block in $\R$}\}.\]
Observe that $ks|R| \leq k^3s^2\l$ vertices in $V$ are in a block in $\R$ and hence $|Q| \leq k^3s^2\l$.

Let $\P'=\R \cup \bigcup_{i \in [t] \setminus Q}(\P \cap \B_i)$ and note that $\P'$ is a partial parallel class in $(V,\B)$. So $T_s(\P')=(T_s(\P) \cup R) \setminus Q$ and $T_0(\P')=[t] \setminus T_s(\P')$. Thus $\t_1(\P')=\cdots=\t_{s-1}(\P') = 0$ and $\t_0(\P') \leq  \t_0(\P)+|Q|$. Furthermore, $T_0(\P)$ contains at most $\epsilon^* t$ indices of full orbits and, by Lemma~\ref{L:fewShort}(i), at most $2\l\sqrt{k}$ indices of short orbits. From this it follows that
\begin{equation}\label{E:dPbound}
d(\P') = (s-1)\t_0(\P') \leq  (s-1)(\t_0(\P)+|Q|) <\mfrac{\e t}{2(k+1)}+O(1) \ll \mfrac{\e t}{k+1}.
\end{equation}

Any block in $\B$ that was $\P$-good but is $\P'$-bad must intersect two of the $s$ blocks in $\P' \cap \B_i$ for some $i \in R$. For each $i \in R$, at most $k^2 \l\binom{s}{2}$ blocks in $\B$ intersect two of the blocks in $\P' \cap \B_i$. So at most $k^2 \l\binom{s}{2}|R|\leq k^4s\l^2\binom{s}{2}$ blocks in $\B$ were $\P$-good but are $\P'$-bad. Thus, for each $i \in T_0(\P')$, because $i \notin R$ and hence less than $\frac{1}{2}st$ blocks in $\B_i$ were $\P$-bad, the number of $\P'$-bad blocks in $\B_i$ is less than $\tfrac{1}{2}st + k^4s\l^2\binom{s}{2}$. Now $t \leq \frac{\l(v-1)}{k(k-1)}+2\l\sqrt{k}$ by Lemma~\ref{L:fewShort}(ii) and hence $st \leq \frac{v}{k}+O(1)$. So, since $|\B_i|\geq \frac{v}{k}$, more than $\frac{v}{k}-\tfrac{1}{2}st - k^4s\l^2\binom{s}{2} \geq \frac{1}{2}st-O(1)$ blocks in $\B_i$ are $\P'$-good. Thus $\P'$ satisfies the conditions of Lemma~\ref{L:greedy} because
\[k^2(ks-k+1)(d(\P')-1) < \e k^2 s t < \tfrac{1}{4}st \ll \tfrac{1}{2}st-O(1)\]
where the first inequality follows by \eqref{E:dPbound} because $ks-k+1<s(k+1)$ and the second follows because $\e < \frac{1}{4k^2}$. Thus, by applying Lemma~\ref{L:greedy} to $\P'$, there is a partial parallel class $\P''$ of $(V,\B)$ such that $\t_s(\P'')=t-\t_{s-1}(\P'')$ and
\[\t_{s-1}(\P'') \leq (k+1)d(\P') < \e t\]
where the last inequality follows by \eqref{E:dPbound}.
\end{proof}

Note that in the special case where $\l$ divides $k-1$, the partial parallel class given by Theorem~\ref{thm:asympt} uses all but at most $\epsilon kt+1$ points of the design.

\vspace{0.3cm}
\noindent{\bf Proof of Theorem \ref{thm:main theorem}}\ \ This follows directly from Theorem \ref{thm:asympt}, noting that $s \geq 2$ because $k \geq 2\l+1$. \qed

\section{Concluding remarks}

A $(v,k,\lambda)$-DDF necessarily has $1\leq \lambda \leq k-1$ apart from the trivial case of a $(k,k,k)$-DDF (see \cite{burattik-1}). Theorem \ref{thm:main theorem} requires $1 \leq \l \leq (k-1)/2$. It is natural to ask whether it is possible to relax this condition. We make the following conjecture.

\begin{Conjecture}\label{con:any k any lambda}
Let $k$ and $\l$ be fixed positive integers such that $k \geq \lambda+1$. There exists an integer $v_0$ such that, for any cyclic $(v,k,\lambda)$-design with $v\geq v_0$, it is always possible to choose one block from each block orbit so that the chosen blocks are pairwise disjoint.
\end{Conjecture}

Compared with Conjecture \ref{conj:any k}, Conjecture \ref{con:any k any lambda} is stated for sufficiently large $v$. This is from the observation that the union of $\lambda$ copies of a $(k(k-1)+1,k,1)$-CDF forms a $(k(k-1)+1,k,\lambda)$-CDF which yields a cyclic $(k(k-1)+1,k,\lambda)$-design without short orbits. Note that a $(k(k-1)+1,k,1)$-CDF is often called a {\em cyclic difference set} (see \cite{baumert}) and it generates a symmetric design, any two blocks of which must intersect in one point. Thus the resulting cyclic $(k(k-1)+1,k,\lambda)$-design cannot be generated by a DDF.

Actually Nov\'{a}k made a stronger conjecture on cyclic STS$(v)$ than Conjecture~\ref{weak conjecture} in 1974. A $(v,3,1)$-DDF for $v\equiv 1\pmod{6}$ is called \emph{symmetric} if its base blocks can be chosen in such a way that for any nonzero $x$ of $\mathbb Z_v$, at most one of $x$ and its complement $v-x$ occurs in the base blocks and no base block contains zero.

\begin{Conjecture}\textup{(Nov\'{a}k, 1974) \cite{novak}}\label{strong conjecture}
Every cyclic STS$(v)$ with $v\equiv 1 \pmod{6}$ is generated by a symmetric $(v,3,1)$-DDF.
\end{Conjecture}

So far it is only known that Conjecture \ref{strong conjecture} holds for all $v\equiv 1\pmod{6}$ and $v\leq 61$ (see \cite[Theorem 22.3]{triplesystem}).

Finally we remark that in a recent paper \cite{bbgrt} a new concept of ``doubly disjoint difference family" was introduced to establish a composition construction for resolvable difference families. Roughly speaking, if we take $k=3$ and $\l=1$ in Theorem \ref{thm:asympt}, then the induced cyclic difference family ``almost" forms a doubly disjoint difference family.

\section*{Acknowledgments}

Research for this paper was carried out while the first and third authors were visiting Monash University. They express their sincere thanks to the School of Mathematics at Monash University for its kind hospitality.

\end{document}